\documentclass[12pt]{article}
\usepackage{graphicx} % Required for inserting images
\usepackage{amsmath}  % Required for math environments
\usepackage{amsfonts} % Required for math fonts
\usepackage{amssymb}  % Required for math symbols
\usepackage{hyperref} % Required for hyperlinks
\usepackage[inline]{enumitem} % Required for enumerated lists customization
\usepackage{geometry} % Required for adjusting page dimensions
\usepackage{theme}
\usepackage{mathtools}
\usepackage[numbers]{natbib}
\usepackage{shortcuts}
\usepackage{algorithm}
\usepackage{algpseudocode}
\usepackage[utf8]{inputenc}
\usepackage[T1]{fontenc}
\usepackage{amsmath}
\usepackage{amssymb}
\setlist[enumerate]{itemsep=0mm}
\usepackage{amsfonts}
\usepackage{graphicx}
\usepackage{hyperref}
\usepackage{color}
\usepackage{ulem}

\newcommand{\R}{\mathbb{R}}
\newcommand{\cov}{\text{cov}}
\newcommand{\N}{\mathbb{N}}
\newcommand{\HS}{\mathcal{H}}
\newcommand{\Cclass}{\mathcal{C}(\HS)}
\newcommand{\CCclass}{\mathcal{CC}(\HS)}

\newcommand{\X}{\mathbb{X}}

\title{General reproducing properties in RKHS with application to derivative and integral operators}
\author{
    Fatima-Zahrae El-Boukkouri\thanks{Institut de Mathématiques de Toulouse, Universit\'e de Toulouse, INSA, Toulouse, 31077, France, \texttt{el-boukkouri@insa-toulouse.fr}} 
    \and
    Josselin Garnier\thanks{CMAP, CNRS, Ecole polytechnique, Institut Polytechnique de Paris, Palaiseau, 91120, France, \texttt{josselin.garnier@polytechnique.edu}} \and
    Olivier Roustant\thanks{Institut de Mathématiques de Toulouse, Universit\'e de Toulouse, INSA, Toulouse, 31077, France, \texttt{roustant@insa-toulouse.fr}} 
}
\begin{document}
\maketitle

\begin{abstract}
Dans cet article, nous considérons la propriété reproduisante dans les espaces de Hilbert à noyaux reproduisants (RKHS). Nous établissons une propriété de reproduction pour l'adhérence de la classe des combinaisons d'opérateurs de composition sous des conditions minimales. Cela nous permet de revisiter les conditions suffisantes pour que la propriété de reproduction soit valable pour l'opérateur dérivé, ainsi que pour l'existence de la fonction mean embedding. Ces résultats donnent un cadre d'application du théorème du représentant pour les algorithmes d'apprentissage régularisés qui impliquent des données sur les valeurs de fonctions, les gradients ou tout autre opérateur de la classe considérée.
\end{abstract}

\begin{abstract}
In this paper, we consider the reproducing property in Reproducing Kernel Hilbert Spaces (RKHS). We establish a reproducing property for the closure of the class of combinations of composition operators under minimal conditions. This allows to revisit the sufficient conditions for the reproducing property to hold for the derivative operator, as well as for the existence of the mean embedding function. These results provide a framework of application of the representer theorem for regularized learning algorithms that involve data for function values, gradients, or any other operator from the considered class.
\end{abstract}

\section*{Introduction}
Machine learning algorithms often involve penalized regression problems of the form
$$ \underset{h \in \HS}{\min} \sum_{i=1}^n (Lh (x_i) - y_i)^2 + \lambda \Vert h \Vert^2. $$
Here $(x_i, y_i)_{1 \leq i \leq n}$ is a dataset where $x_i$ belong to some set $\X$ and $y_i \in \R$, $\lambda \in \R_+$ is a penalty coefficient, $\HS$ is a Hilbert space of real-valued functions on $\X$, and $L$ is a linear operator on $\HS$. When $\HS$ is a reproducing kernel Hilbert space (RKHS), it turns out that the solution lives in a finite-dimensional space, a famous result known as ``representer theorem'' \cite{Reproducing_kernel, Representer_Theorem}.

A key property to prove this result is the generalized  ``reproducing property''
\begin{equation} \label{eq:LtildeIntro}
\forall x \in \X, \, \exists  \widetilde{L}(x) \in \HS \mbox{ s.t. } \forall h \in \HS, \,  Lh(x) = \langle h, \widetilde{L}(x) \rangle   .
\end{equation}
Note that, when $\widetilde{L}(x)$ exists, it is necessarily unique.
When $L$ is the identity operator, we have $\widetilde{L}(x) = K(x, .)$ where $K$ is the kernel associated to $\HS$,
\begin{equation*} 
\forall x \in \X, \forall h \in \HS: \qquad h(x) = \langle h, K(x, .) \rangle    ,
\end{equation*}
which is the original reproducing property of RKHS \cite{Reproducing_kernel}. This immediately extends to finite linear combinations, i.e. when $L$ has the form
$$
Lf(x) = \sum_{i=1}^{q} \alpha_{i} f(v_i(x)) ,
$$
where $\alpha_i \in \R$ and $v_i(x) \in \X$. In that case, $\widetilde{L}(x) = \sum_{i=1}^{q} \alpha_{i} K(v_{i}(x),.).$ However, the generalization to more complex operators such as derivative or integral operators is not straightforward as it involves a passage to the limit. 

In this paper, we consider a broad class of operators corresponding to limits of linear combinations and we give a necessary and sufficient condition on $K$ for \eqref{eq:LtildeIntro} to hold under the assumption that  \eqref{eq:LtildeIntro} remains true when passing to the limit. This allows us to revisit the reproducing properties in RKHS for derivative and integral operators.
Firstly, we focus on the derivative operator, and show that the reproducing property holds if the cross derivative of the kernel exists and is continuous on the diagonal of $\X \times \X$. Thereby we retrieve the result presented in \cite{christmann2008support,saitoh_sawano_book} by a different approach.
We prove that this condition is less restrictive than the well-known $C^2$ class condition \cite{Zhou, L2_kernel} and we exhibit a counterexample.

Secondly, we consider the mean embedding in RKHS within the frame of (improper) Riemann integrals. We prove that it is properly defined if the kernel is absolutely integrable.
This is similar to the result of \cite{carmeli2006vector, barp2024targeted, oates2022minimum} obtained for Lebesgue integrals.
Then, we go a step further and show that, under the same condition, the reproducing property holds, which does not seem to be reported in the literature. We show that it is less restrictive than the condition that $x \mapsto \sqrt{K(x,x)}$ is integrable \cite{Mean_embedding, muandet2017kernel} and we exhibit a counterexample.

\section{Reproducing property for linear operators}\label{sec:linearoperators}
A RKHS $\HS$ is a Hilbert space of real-valued functions defined on a set $\X$ such that for all \( x \in \X \), the evaluation $h \in \HS \mapsto h(x) $ is continuous. Giving a RKHS is equivalent to giving a positive semidefinite function, or kernel, $K$.
The link between $\HS$ and $K$ is given by the reproducing property: 
\[
 \forall f \in \HS, \forall x \in \X: \qquad f(x) = \langle f, K(x, \cdot) \rangle_{\mathcal{H}}   .
\]
This implies, by choosing $f=K(y, .)$, that  
$K(x,y) = \langle K(x, .), K(y, .) \rangle_\HS$ for all $x, y \in \X$.
\subsection{Main result}
We first recall the definition of so-called combination of composition operators, introduced in \cite{composition_op}.

\begin{definition}
Let $v: \X \to \X$ be an arbitrary function. 
The composition operator \( T_{v} \), with function \(v\), is defined by:
\[
T_{v}: \mathcal{H} \to \mathcal{F}(\X , \R), \quad T_{v}(f):= f \circ v    .
\]
We call a combination of composition operators with functions $v_i$ and weights $\alpha_i \in \R$ for $ 1 \leq i \leq q $ the operator:
\[ 
L = \sum_{i=1}^q \alpha_i T_{v_i}   .
\]
The class of combination of composition operators on functions on $\mathcal{H}$ is denoted $ \mathcal{CC(\HS)}$.\\
\end{definition}

Let $L = \sum_{i=1}^{q} \alpha_{i} T_{v_i} \in \CCclass$. Notice that for all $x \in \X$, $Lf(x) = \sum_{i=1}^{q} \alpha_{i} f(v_i(x))$. Thus, the linear form $f \in \HS \mapsto Lf(x)$ is continuous as a linear combination of evaluations of $\HS$. Following \cite[\S 4.4.]{berlinet2011reproducing}, we will denote by $\widetilde{L}(x) \in \HS$ its representer, i.e. the unique element of $\HS$ verifying
\begin{equation} \label{eq:Ltilde}
 Lf(x) = \langle f, \widetilde{L}(x) \rangle   . 
\end{equation}
We have explicitly 
\begin{equation} \label{eq:L_ntilde_expl}
    \widetilde{L}(x) = \sum_{i=1}^{q} \alpha_{i} K(v_{i}(x), .). 
\end{equation}
Mind that $\widetilde{L}(x) \neq L(K(x, .)) = \sum_{i=1}^{q} \alpha_{i} K(x, v_{i}(.)).$\\

The class $\CCclass$ is limited to finite linear combinations. We now define a natural extension of this class by considering its closure with respect to pointwise convergence.

\begin{definition} \label{def:Cclass}
We define the class of operators  
\[\Cclass= \big\{ L: \mathcal{H} \to \mathcal{F}(\X , \R): \exists (L_n)_{n \in \mathbb{N}} \in \mathcal{CC(\HS)}^{\N} \mbox{ s.t. } L(f) = \underset{n \to +\infty}{\lim} L_n(f) \, \forall f \in \mathcal{H}   \big\}   , \]  
where the limit is in the pointwise sense.
\end{definition}

If $L \in \Cclass$, writing $L_n = \sum_{i=1}^{q_n} \alpha_{i,n} T_{v_{i,n}}$ (where $q_n \in \N, \alpha_{i,n} \in \R, v_{i,n} \in \mathcal{F}({\X, \X}$)), we have for all $f \in \HS $ and all $x \in \X$,
$$
L(f)(x) = \underset{n \to +\infty}{\lim} \sum_{i=1}^{q_n} \alpha_{i,n} f(v_{i,n}(x)) .
$$
As an example, when $\X = \R$, the class \(\Cclass\) includes derivative and integral operators: 
\begin{itemize}
    \item If the functions of $\HS$ are differentiable, the derivative operator $L: f \mapsto f'$ is written as  $\underset{n \to +\infty}{\lim} L_n$ where $L_n = nT_{v_{2,n}}- n T_{v_{1, n}} $ with $v_{2,n}: 
    x \mapsto x + \frac{1}{n}$ and $v_{1, n}: 
    x \mapsto x $.  Here $\widetilde{L_n}(x) = n K(v_{2,n}(x), .) - n K(v_{1,n}(x), .)$.
    \item If the functions of $\HS$ are continuous, the averaging operator  $L: f \mapsto  x^{-1} \int_0^x f(t) dt$ (with $L(f)(0)=f(0)$) is written as $\underset{n \to +\infty}{\lim} L_n$ where $L_n = \sum_{i=1}^n \frac{1}{n} T_{v_{i,n}}$ with  $v_{i,n}: 
    x \mapsto  \frac{i}{n} x$.
   Here $\widetilde{L_n}(x) = \sum_{i=1}^n \frac{1}{n} K(v_{i,n}(x), .)$. \\
\end{itemize}

We now recall the Loève criterion for convergence of sequences in Hilbert spaces.

\begin{proposition}[Loève criterion]
    Let $\HS$ be a Hilbert space with inner product \(\langle ., . \rangle\).
    Let $(h_n)_{n \in \N}$ be a sequence of \(\mathcal{H}\). The following statements are equivalent: 
    \begin{itemize}
        \item[(i)] \((h_n)_{n \in \N}\) converges in \(\mathcal{H}\).
        \item[(ii)] The double sequence $(\langle h_n, h_m \rangle)_{n, m \in \N}$ has a finite limit as \(n, m\) tend to \(+\infty\).
    \end{itemize}
    \label{loève}
\end{proposition}

In the whole paper the convergence of double sequences $(u_{n,m})_{n,m \in \N}$ -- such as the one appearing in (ii), is in the uniform sense (also known as Pringsheim sense \cite{tripathy2005convergent}):
if $c$ is its limit, $\forall \varepsilon > 0, \exists n_0 \in \N, \forall n,m \geq n_0, \vert u_{n,m} - c \vert \leq \varepsilon$.

\begin{proof}
The proof can be found in \cite[\S 3.5]{cramer} when $\HS$ is a space of square-integrable random variables. For completeness, we recall here the main arguments.\\
The direct sense (i) $\Rightarrow$ (ii) is obvious by continuity of the scalar product.\\
Let us now prove (ii) $\Rightarrow$ (i).
Denote by $c$ the limit of $(\langle h_n, h_m \rangle)$ when $n, m$ tend to infinity.
We have, for all $n, m \in \mathbb{N}$,
\begin{equation}
    \|h_n - h_m \|^2 =  \langle h_n, h_n\rangle +  \langle h_m, h_m\rangle -2  \langle h_n, h_m\rangle    .
    \label{eq:cauchy}
\end{equation}
 This implies that  $ \|h_n - h_m \|^2 \to c + c - 2c = 0$ when $ n,m \to + \infty$ . 
This proves that $(h_n)$ is a Cauchy sequence, and thus converges in $\HS$.
\end{proof}

For all operator $L$ and for all function $F: \X \times \X \to \R$ denote 
\[
L^\ell F: \X \times \X \to \mathbb{R}, \quad (x, y) \mapsto L(F(\cdot, y))(x),
\]
and
\[
L^r F: \X \times \X \to \R, \quad (x, y) \mapsto L(F(x, \cdot))(y).
\]
Notice that with these notations, for all $x, y \in \X$ we have $(L^\ell F)(., y) = L (F(., y))$ and $(L^r F)(x, .) = L (F(x, .))$. We will also simply write $L^\ell L^r$ instead of $L^\ell \mathord{o} L^r$.\\ 

\begin{theorem} Let $\mathcal{H}$ be a RKHS with kernel $K$.
Let $(L_n)_{n \in \N}$ be a sequence of $\CCclass$. 
We denote by $\widetilde{L_n}(x) \in \HS$ the representer of the linear form $f \in \HS \mapsto L_nf(x)$ (see \eqref{eq:Ltilde}).
The following conditions are equivalent: 
\begin{itemize}
    \item[(i)] $(L_n^\ell L_m^r K(x,x))_{n, m}$ converges when $n, m$ tend  to $+\infty$ for all $x \in \X$.
    \item[(ii)] $(\widetilde{L_n}(x))_{n \in \N}$ converges in $\mathcal{H}$ for all $x \in \X $. 
\end{itemize}
In that case, denote $\widetilde{L}(x):= \underset{n \to +\infty}{\lim} \widetilde{L_n}(x) \in \HS$. Then, for all $f \in \HS$, $L_n f$ converges pointwise. Define the operator $L \in \Cclass$ by $Lf(x) = \underset{n \to +\infty}{\lim} L_n f(x)$ (for all $f \in \HS$, $x \in \X$).
Then, for all $x \in \X$ the mapping $f \in \HS \mapsto Lf(x)$ is continuous and the reproducing property holds: 
    \[ 
    \forall x \in \X, \forall f \in \HS: \qquad Lf(x) = \langle f, \widetilde{L}(x) \rangle.
    \]
    \label{THM Property}
Finally, $\Vert \widetilde{L}(x) \Vert^2 = L^\ell L^r K(x,x)$. 
\end{theorem}

\begin{proof}
Let $x \in \X$. We first prove that
\begin{equation} \label{eq:LLK}
\langle \widetilde{L_n}(x), \widetilde{L_m}(x) \rangle = L_n^\ell L_m^r K(x, x).
\end{equation}

Let $y \in \X$. 
Denote $F_m = L_m^r K$. Using the definition (\ref{eq:Ltilde}) of $\widetilde{L_m}(y)$ (with $f = K(x, .)$), we have
\begin{equation} \label{eq:definitionF_m}
 F_m (x, y) = (L_m K(x, .))(y) 
= \langle K(x, .), \widetilde{L_m}(y) \rangle 
= \widetilde{L_m}(y)(x)   .
\end{equation}

Thus, $F_m(., y) = \widetilde{L_m}(y)$ by uniqueness. 
Similarly, noting that $F_m(., y) \in \HS$ and using the definition (\ref{eq:Ltilde}) of $\widetilde{L_n}(x)$ (with $f = F_m(., y)$), we have
\begin{equation}\label{eq:LnFm}
L_n^\ell F_m (x, y) = (L_n F_m(., y))(x) 
= \langle F_m(., y), \widetilde{L_n}(x) \rangle.     
\end{equation}

Finally, we get
$ L_n^\ell L^r_m K (x, y) 
= \langle \widetilde{L_m}(y), \widetilde{L_n}(x) \rangle$, which gives \eqref{eq:LLK} when $y=x$.

From Equation \eqref{eq:LLK},  by the Loève criterion (Proposition \ref{loève}), the sequence 
$\widetilde{L_n}(x)$ converges in the Hilbert space $\HS$ for all $x \in \X$ if and only if 
$$
\lim_{n,m \to \infty} L_n^\ell L_m^r K (x, x)
$$
exists at each point \( x \in \X \).
This proves that (i) and (ii) are equivalent.

Suppose now that (i) or (ii) are verified and denote $\widetilde{L}(x):= \underset{n \to +\infty}{\lim} \widetilde{L_n}(x) \in \HS$.
Writing the equality (\ref{eq:Ltilde}) defining $\widetilde{L_n}(x)$, for a given $f \in \HS$ and $x \in \X$, 
$$
L_n f(x) = \langle f, \widetilde{L_n}(x) \rangle  ,
$$
and taking the limit when $n$ tends to infinity, we obtain that $L_n f(x)$ converges. Denoting by $Lf(x)$ its limit, we thus have:
$$ Lf(x) = \langle f, \widetilde{L}(x) \rangle. $$
This ensures that the mapping $f \in \HS \mapsto Lf(x)$ is continuous.
Finally, by taking the limit in \eqref{eq:LLK}, we get: 
$$ \Vert \widetilde{L}(x) \Vert^2 = L^\ell L^r K(x,x).$$
Indeed, let $F = L^rK$ and $x,y \in \X$. By taking the limit in \eqref{eq:definitionF_m}, we get $F(.,y) = \widetilde{L}(y) \in \HS$. Then, taking the limit in \eqref{eq:LnFm}, we obtain
$L^\ell F (x, y) = \langle F(., y), \widetilde{L}(x) \rangle$, or equivalently, 
$ \langle \widetilde{L}(x),  \widetilde{L}(y)\rangle = L^\ell L^r K(x,y).$
\end{proof}

\section{Derivative Reproducing Property}
\label{sec:derivative}
The reproducing property for the derivative operator has been extensively studied in the literature under various conditions. For example, in \cite{Zhou}, this result is established for Mercer kernels of class \( C^2 \). A weaker condition is given in \cite[\S 2]{saitoh_sawano_book} and \cite[Cor 4.36]{christmann2008support}, namely that the cross derivative \(\frac{\partial^2 K}{\partial x_1 \partial x_2}\) exists and is continuous on \(\X \times \X\). In what follows, we obtain a similar result, demanding the continuity at the neighborhood of the diagonal of $\X \times \X$, by following a different route with the class of compositions of operators. We argue that the continuity assumption, sometimes omitted \cite{gaetan2010second, cramer}, is important when the RKHS is not supposed \textit{a priori} to contain differentiable functions.

\begin{theorem}     \label{thm:derivative}
Let $\HS$ be a RKHS of real-valued functions defined on a non-empty open set $\X \subset \R$, with reproducing kernel $K$. If $K$ is of class $C^1$ and $\frac{\partial^2 K}{\partial x_1 \partial x_2}$ exists and is continuous on a neighborhood of the diagonal, then
\begin{itemize}
    \item[(a)] For all $x \in \X$,  $\frac{\partial K}{\partial x_1}(x, .) $ is in $\HS$, and $\Vert \frac{\partial K}{\partial x_1}(x, .) \Vert^2 = \frac{\partial^2 K}{\partial x_1 \partial x_2}(x,x)$.
    \item[(b)] All functions of $\HS$ are differentiable, the mapping $f \in \HS \mapsto f'(x)$ is continuous, and the derivative reproducing property holds: 
    \[ 
    \forall x \in \X, \forall f \in \HS: \qquad 
    f'(x) = 
    \langle f,  \frac{\partial K}{\partial x_1}(x,.) \rangle .
    \]
 \end{itemize}
\end{theorem}

We will need the well-known fact that convergence in $\HS$ implies pointwise convergence (see \cite{berlinet2011reproducing}).
\begin{lemma} \label{lemma:RKHSandPointwiseConvergence}
Let $\HS$ a RKHS on $\X$ with kernel $K$, and let $(h_x)_{x \in \X}$ a family of functions of $\HS$. If for some $x_0 \in \X$, $h_x$ converges in $\HS$ when $x \to x_0$, then $h_x$ converges pointwise to the same limit. 
\end{lemma}

\begin{proof}
Let $h_\infty \in \HS$ be the limit of $h_x$ when $x \to x_0$. Let $y \in \X$.
By the reproducing property and the Cauchy-Schwarz inequality, 
$$ \vert h_x(y) - h_\infty(y) \vert
= \vert \langle  h_x - h_\infty, K(y, .) \rangle
\vert \leq 
\Vert h_x - h_\infty \Vert
\Vert K(y, .) \Vert.
$$
The result follows.
\end{proof}

\begin{proof}[Proof of Theorem \ref{thm:derivative}]
Suppose that $K$ is of class $C^1$ and $\frac{\partial^2 K}{\partial x_1 \partial x_2}$ exists and is continuous on a neighborhood of the diagonal.\\
Let $(u_n)$ be a real-valued sequence converging to $0$ as $n$ tends to infinity with $u_n\neq~0$ $ \forall n \in \N$.\\
Let $n \in \N$ and $x \in \X$ and consider the operator $L_n \in \CCclass$ defined by 
    \[
    L_n = \frac{T_{v_{2,n}} - T_{v_{1,n}}}{u_n}, \quad \text{where} \quad v_{2,n}(t) = t +u_n, \quad v_{1,n}(t) = t.
    \] 
    We have
    \[
    L_n(f)(x) = \frac{f(x+u_n) - f(x)}{u_n} = \langle f, \widetilde{L_n}(x) \rangle   ,
    \]
    with
    \[ \widetilde{L_n}(x) = \frac{K(x + u_n, .) -K(x,.)}{u_n}.
    \]
    Notice that
\begin{eqnarray*}
L_n^\ell L_m^r K(x,x) 
=\frac{K(x + u_n, x + u_m) - K(x+u_n, x) - K(x, x+u_m) + K(x,x)}{u_n u_m} .
\end{eqnarray*}
   As \(
   \frac{\partial^2 K}{\partial x_1 \partial x_2}
   \) exists and is continuous on a neighborhood of the diagonal, the function $(u', v') \mapsto \frac{\partial^2 K}{\partial x_1 \partial x_2}(x + u',x + v')$ is continuous on a neighborhood of $(0, 0)$. This justifies that
   there exists $n_0 \in \N$ such that for all $n,m \geq n_0$, 
    \begin{eqnarray*}
    L_n^\ell L_m^r K(x,x) 
    = \frac{\int_0^{u_n} \int_0^{u_m}\frac{\partial^2 K}{\partial x_1 \partial x_2}(x + u',x + v')du' dv' }{u_n u_m},
    \end{eqnarray*}
    which can be obtained by two successive integrations with respect to $u'$ and $v'$ respectively. As $ \frac{\partial^2 K}{\partial x_1 \partial x_2}$ is continuous at $(x,x)$, then for a given $\epsilon > 0$, there exists $ r > 0$ such that for all $u',v' \in [-r, r]$, 
    $$ \left| \frac{\partial^2 K}{\partial x_1 \partial x_2}( x + u', x+ v') - \frac{\partial^2 K}{\partial x_1 \partial x_2}(x,x)  \right| < \epsilon $$
    As the sequence $(u_n)$ converges to $0$, there exists $N_0 \in \N$, such that for all $n \geq N_0$, $u_n \in [-r, r] $. Then for all $n ,m \geq \max(n_0, N_0)$,
\[
        \left| 
         L_n^\ell L_m^r K(x,x) -  \frac{\partial^2 K}{\partial x_1 \partial x_2}(x,x) \right| 
         \leq \left| \int_0^{u_n} \int_0^{u_m} \frac{ \left| \frac{\partial^2 K}{\partial x_1 \partial x_2}(x + u',x + v') - \frac{\partial^2 K}{\partial x_1 \partial x_2}(x,x) \right| du' dv' }{u_n u_m}  \right| 
     \leq \epsilon
\]

    This implies that $(L_n^\ell L_m^r K(x,x))_{n, m}$ converges when $n, m$ tend  to $+\infty$ to $\frac{\partial^2 K}{\partial x_1 \partial x_2}(x ,x )$.\\
    Then using Theorem~\ref{THM Property}, we conclude that  $(\widetilde{L_n}(x))_{n \in \N}$ converges in $\HS$.
    Denoting by $\widetilde{L}(x)$ its limit, we have by Lemma \ref{lemma:RKHSandPointwiseConvergence}, 
    \[  \frac{\partial K}{\partial x_1}(x, .) = \widetilde{L}(x). \] 
    As the limit of $\widetilde{L_n}(x)$ does not depend on the sequence $u_n$, we can conclude that $\frac{K(x + t, .) - K(x, .)}{t} $ converges in $\HS$ when $t$ tends to $0$ to $\frac{\partial K}{\partial x_1}(x, .)$.\\
    To conclude the proof of (a), let us show that $$ \Vert \frac{\partial K}{\partial x_1}(x, .) \Vert^2 = \frac{\partial^2 K}{\partial x_1 \partial x_2}(x,x)
    .$$
    From Theorem~\ref{THM Property}, for all $f \in \HS$, $L_n f$ converges pointwise, and with $Lf(x) = \underset{n \to +\infty}{\lim} L_n f(x)$, we have 
    $$ 
    \Vert \widetilde{L}(x) \Vert^2 = L^\ell L^r K(x,x)  .
    $$
    Recall that $(L_n^\ell L_m^r K(x,x))_{n, m}$ converges to $\frac{\partial^2 K}{\partial x_1 \partial x_2}(x ,x )$ when $n, m$ tend  to $+\infty$. In particular, $(L_n^\ell L_n^r K(x,x))$ converges to $\frac{\partial^2 K}{\partial x_1 \partial x_2}(x ,x )$ when $n$ tends to $+\infty$. This implies that $  L^\ell L^r K(x,x) =  \frac{\partial^2 K}{\partial x_1 \partial x_2}(x ,x) $
    And thus, $$ \Vert \frac{\partial K}{\partial x_1}(x, .) \Vert^2 = \frac{\partial^2 K}{\partial x_1 \partial x_2}(x,x)  .
    $$
    
    It remains to show (b). Let $f \in \HS$ and $x \in \X$. We know that $\frac{\partial K}{\partial x_1}(x,.)$ is in $\HS$.
    We have
    $$ 
    \frac{f(x + t) -f(x)}{t} =  \langle f, \frac{K(x + t, .) - K(x, .)}{t} \rangle .
    $$
    As $\frac{K(x + t, .) - K(x, .)}{t} $ converges to $\frac{\partial K}{\partial x_1}(x, .)$ in $\HS$ when $t$ tends to $0$ , we obtain that $f$ is differentiable at $x$ and 
    $$  f'(x) = 
    \langle f,  \frac{\partial K}{\partial x_1}(x,.) \rangle .
    $$
    This implies that the mapping $f \in \HS \mapsto f'(x)$ is continuous.
\end{proof}

\begin{remark}
Theorem \ref{thm:derivative} implies that if $Y$ is a centered second-order random process with covariance function $K$, then $Y$ is everywhere differentiable in quadratic mean (q.m.), the second-order cross derivative $\frac{\partial^2 K}{\partial x_1 \partial x_2}$ exists everywhere and $\cov(Y'(s), Y'(t)) = \frac{\partial^2 K}{\partial x_1 \partial x_2}(s, t)$ \cite{gaetan2010second}.
\end{remark}
\begin{remark}
It may be not sufficient to assume only the existence of the cross derivative on the diagonal (without assuming its continuity on a neighborhood of the diagonal). Indeed, in the proof of Theorem  \ref{thm:derivative}, we can write 
\begin{eqnarray*}
L_n^\ell L_m^r K(x,x) 
= \frac{\frac{K(x + u_n , x + u_m) - K(x + u_n, x)}{u_m} - \frac{K(x , x + u_m) - K(x , x)}{u_m}}{u_n}   ,
\end{eqnarray*}
which leads to the fact that 
$\underset{n \to +\infty}{\lim} \underset{m \to +\infty}{\lim} L_n^\ell L_m^r K(x,x) $ exists. Here, the limit is taken sequentially firstly in $m$ and secondly in $n$. This does not imply that the limit of this double sequence exists when $n, m$ tend to infinity, which is required by the Loève criterion (Proposition \ref{loève}). \\
Finally, notice that the continuity of the cross derivative can be omitted if one assumes in addition from the outset that all functions in \(\HS\) are differentiable, as proved in \cite[Lemma 4]{barp2024targeted}.

\end{remark}

Theorem \ref{thm:derivative} gives a sufficient condition on $K$ for a derivative reproducing property to hold, namely the existence of the cross-derivative and its continuity on a neighborhood of the diagonal. This is
less restrictive than the $C^2$ condition on $K$ found in the literature. Indeed, the following example exhibits a kernel $K$ that is not of class $C^2$ but such that $\frac{\partial^2 K}{\partial x_1 \partial x_2}$ exists and is continuous.

\begin{example}[Non-$C^2$ Kernel with finite cross derivative]
Consider the function:
\[
f: \R \to \R, \quad f(x) = 
\begin{cases} 
x^4 \sin\left(\frac{1}{x}\right) & \text{if } x \neq 0, \\
0 & \text{if } x = 0.
\end{cases}
\]

We notice that \( f \) is twice differentiable on \( (-\infty, 0) \) and \( (0, \infty) \).
For \( x \in \R \backslash \{0\} \) , the first and second derivatives are:
\[
f'(x) = 4x^3 \sin\left(\frac{1}{x}\right) - x^2 \cos\left(\frac{1}{x}\right)
\]
and
\[
f''(x) = 12x^2 \sin\left(\frac{1}{x}\right) - 6x \cos\left(\frac{1}{x}\right) - \sin\left(\frac{1}{x}\right).
\]
Observe that
\[
 \frac{f(x) - f(0)}{x} =  \frac{x^4 \sin\left(\frac{1}{x}\right)}{x} = x^3 \sin\left(\frac{1}{x}\right)   ,
\]
which proves that \( f \) is differentiable at \( x = 0 \) with \( f'(0) = 0 \).
Similarly, 
\[
 \frac{f'(x) - f'(0)}{x} = \frac{f'(x)}{x} =  \left( 4x^2 \sin\left(\frac{1}{x}\right) - x \cos\left(\frac{1}{x}\right) \right)   ,
\]
which shows that \( f \) is twice differentiable at \( x = 0 \) with \( f''(0) = 0 \).
We conclude that \( f \) is twice differentiable on \( \R \), with the second derivative given by:
\[
f''(x) = 
\begin{cases}
12x^2 \sin\left(\frac{1}{x}\right) - 6x \cos\left(\frac{1}{x}\right) - \sin\left(\frac{1}{x}\right) & \text{if } x \neq 0, \\
0 & \text{if } x = 0.
\end{cases}
\]
However, $f''$ has no limit at $0$, due to the term $\sin(\frac{1}{x})$, which means that $f$ is not $C^2$.\\
Now, consider the rank-one kernel $K$: 
\[
K: \R\times\R \to \R, \quad K(x,y) = f(x)f(y).
\]
\( K \) is not of class \( C^2 \) since for all $y \in \R$, the function $x \mapsto \frac{\partial^2 K}{\partial x_1^2}(x,y) = f''(x) f(y)$ is not continuous at $x=0$. Nevertheless, the cross derivative $ \frac{\partial^2 K}{\partial x_1 \partial x_2}(x, y) = f'(x) f'(y) $ exists for all \( x, y \in \mathbb{R} \) and is continuous. 
\end{example}

We now extend the previous result on the derivative reproducing property to partial derivatives. We will identify the conditions required on the kernel \(K : \mathbb{R}^d \times \mathbb{R}^d \to \mathbb{R}\), with $d \in \N^*$, for the reproducing property to hold.

\section{Mean Embedding in RKHS}
\label{sec:meanembedding}
In this subsection, we consider a random variable $X$ defined on $\X \subset \R$ with probability density function $p$ and a RKHS $\mathcal{H}$ of  real-valued  functions  defined  on $\X$ with kernel $K$. We seek to establish the minimal condition under which there exists a function $\mu_p \in \HS$ satisfying: 
\[ 
\mathbb{E}_{X \sim p} f(X) = \langle f , \mu_p \rangle, \quad \forall f \in \HS    .
\]
The function $\mu_p$ is called the mean embedding of $p$.

Fukumizu et al. \cite{Mean_embedding} proved that $\mu_p$ exists if $\mathbb{E}_{X \sim p} \sqrt{K(X,X)} < \infty $ in the sense of the Lebesgue integral. However, the function $\mu_p$ exists under a less restrictive condition as we will demonstrate.\\
In what follows, our results will use the integral in the Riemann sense. To demonstrate it, we will need to use some notions related to the improper Riemann integral \cite{Zorich}.

\begin{definition}[Admissible Set] \label{def:admissible_set}
A set \( E \subset \mathbb{R}^d \) is admissible if it is bounded and its boundary \( \partial E \) has Lebesgue measure zero.
The integral of a function \( f \) over \(E \) is defined as:
\[
\int_E f(x) \, dx = \int_I f(x) \mathbb{1}_E(x) \, dx,
\]
where \( I \) is some interval in \( \mathbb{R}^d \) (i.e. a set of the form $\{ x \in \R^d, \, a^i\leq x_i \leq b^i, i=1,\ldots,d\}$, for $a^i<b^i\in \mathbb{R}$) such that \( E\subset I \).
\end{definition}
If the integral on the right-hand side of this equality exists, then we say that $f$ is Riemann integrable over $E$. Lebesgue’s criterion \cite[\S 11.1.2, Thm 1]{Zorich} 
for the existence of the Riemann integral over an interval implies that a function $f:E\to \R$ is integrable over an admissible set $E$ if and only if it is bounded and continuous at almost all points of~$E$ \cite[\S 11.2.2, Thm 1]{Zorich}.

\begin{definition}[Exhaustion]
An exhaustion of a set \( E \subset \mathbb{R}^d \) is a sequence of admissible sets \((E_s)_{s \in \N} \) such that:
\[
E_s \subset E_{s+1} \subset E \quad \text{for all } s \in \mathbb{N},
\]
and
\[
\bigcup_{s=1}^\infty E_s = E.
\]
\end{definition}
The notion of exhaustion defines the improper integral over a set $E \subset \R^d$.
\begin{definition}[Riemann improper integral]
Let $(E_s)_{s \in \N}$ be an exhaustion of the set $E$ and suppose the function $f :  E \to \R $ is Riemann integrable on the $E_s$ for all $s \in \N$. If the sequence $ \int_{E_s} f(x) dx$ converges and has a limit independent of the choice of the exhaustion, then $f$ admits an improper integral over $\X$ defined by
\[
\int_E f(x) dx = \underset{s \to \infty}{\lim} \int_{E_s} f(x) dx.
\]
\end{definition}

\begin{proposition}[Comparison Test for Improper Integrals, \cite{Zorich}] \label{prop:comp}
Let \( f \) and \( g \) be functions defined on \( E \) and integrable over exactly the same admissible subsets of $E$, and suppose \( |f| \leq g \) on \( E \). If the improper integral
\(
\int_E g(x) \, dx
\)
exists, then the improper integrals
\(
\int_E |f(x)| \, dx \quad \text{and} \quad \int_E f(x) \, dx
\)
also exist.

\end{proposition}
Using improper integrals, we now establish the condition for the existence of the mean embedding \(\mu_p\) in \(\HS\) when \(\X\) is not necessarily an interval. Note that $\X$ can be unbounded and $K$ can be unbounded. \\
\begin{theorem}
Let $\mathcal{H}$ be a RKHS of real-valued functions defined on $\X$, with reproducing kernel $K$.
Assume that $K$ and $p$ are continuous almost everywhere and locally bounded on $\X \times \X$ and $\X$ respectively, i.e. $K$ (resp. $p$) is bounded on all bounded subsets of $\X \times \X$ (resp. $\X$). Assume that the Riemann improper integral $\int_{\X \times \X} |K(x,y)|p(x)p(y)dx \, dy$ exists. Then,
\begin{itemize}
    \item[(a)] The mean embedding $\mu_p : x \in \X \mapsto \int_\X K(x,y) p(y) dy$ exists and is in $\HS$. Moreover,  $\Vert \mu_p \Vert^2 = \iint_{\X \times \X} K(x, y) p(x) p(y) dx \, dy$.
    \item[(b)] 
    For all function $f \in \HS$, the function $fp$ admits an improper integral on $\X$ and the mapping $f \in \HS \mapsto \mathbb{E}_{X \sim p}[f(X)] = \int_\X f(x) p(x) dx$ is continuous. Furthermore, the reproducing property holds: 
    $$ \forall f \in \HS, \qquad \mathbb{E}_{X \sim p}[f(X)] = \langle f, \mu_p \rangle. $$
\end{itemize}
\label{thm:mean_embedding}
\end{theorem}

\begin{proof}
Let $(E_s)_{s\in \N}$ be an exhaustion of $\X$. For all $s \in \N$, as $E_s$ is bounded, there exists an interval $I_s$ in $\R$ such that for all $f \in \HS$,
$$ 
\int_{E_s} f(x) p(x) dx = \int_{I_s} f(x) \mathbb{1}_{E_s}(x) p(x)dx    .
$$
Let $s \in \N$ and consider for all $n \in \N^*$, a partition $(v^s_{i,n})_{0 \leq i \leq n}$ of the interval $I_s$ composed of $n$ nodes. Suppose that the mesh of this partition $\lambda_n$ defined as $\lambda_n = \underset{1 \leq i \leq n}{\max} (v^s_{i,n} - v^s_{i-1,n})$ converges to 0 as $n$ tends to $\infty$.
Let $n \in \N^*$ and consider the following operator in $\CCclass$,
$$ 
L_{s,n} = \sum_{i=1}^n(v^s_{i,n} - v^s_{i-1,n})p(v^s_{i,n})\mathbb{1}_{E_s}(v^s_{i,n}) T_{v^s_{i,n}}    .
$$
Thus, by \eqref{eq:L_ntilde_expl}, 
$$
\forall x \in \X, \qquad \widetilde{L_{s,n}}(x) = \sum_{i=1}^n(v^s_{i,n} - v^s_{i-1,n})p(v^s_{i,n})\mathbb{1}_{E_s}(v^s_{i,n})K(v^s_{i,n}, .)    .
$$

From now on, as $L_{s,n}f(x)$ and $\widetilde{L_{s,n}}(x)$ do not depend on $x$, we will omit $x$ and simply write $L_{s,n}f$ and $\widetilde{L_{s,n}}$.\\ 
Notice that for all $n,m \in \N^*$, and all $x \in \X$,
\begin{align*} 
&L^\ell_{s,n} L^r_{s,m} K(x,x) \\
&
= \sum_{i=1}^n \sum_{j=1}^m (v^s_{i,n} - v^s_{i-1,n})p(v^s_{i,n}) (v^s_{j,m} - v^s_{j-1,m})p(v^s_{j,m})\mathbb{1}_{E_s \times E_s }(v^s_{i,n},v^s_{j,m}) K(v^s_{i,n},v^s_{j,m})  .
\end{align*}

From the assumptions on $K$ and $p$, the function $(x,y) \mapsto K(x,y) p(x)p(y)$ is continuous and 
bounded on the bounded set $E_s \times E_s$. 
Then, by \cite[\S 11.2.2, Thm 1]{Zorich} the Riemann integral  
\( \iint_{E_s \times E_s }K(x, y) p(x)p(y)dx \, dy \) exists.
As the function $x,y \mapsto K(x,y)p(x)p(y)\mathbb{1}_{E_s \times E_s}(x,y)$ is continuous almost everywhere on the interval $I_s$, the double sum above converges as $n, m$ tend  to $+\infty$ to $\int_{I_s \times I_s} K(x,y)p(x)p(y)\mathbb{1}_{E_s \times E_s}(x,y) dx \, dy $, which is equal by definition \ref{def:admissible_set} to $\int_{E_s \times E_s} K(x,y)p(x)p(y) dx \, dy $.
Then using Theorem $\ref{THM Property}$ on the sequence $(L_{s,n})_{n \in  \N}$ , we conclude that $(\widetilde{L_{s,n}})_{n \in \N}$ converges in $\HS$. Denote by $\widetilde{L_s}$ its limit. Furthermore, for all $f \in \HS$, $L_{s,n}f$ converges pointwise with respect to $n$, and we denote $L_s f = \underset{n \to \infty}{\lim} L_{s,n} f$. Finally (still by Theorem \ref{THM Property}),
\begin{equation}
\label{eq:scalarf}
\forall f \in \HS, \qquad L_s f =  \langle f, \widetilde{L_s} \rangle   .
\end{equation}

For all $x \in \X$, the function $y \mapsto K(x,y)p(y)$ is continuous almost everywhere and bounded on the bounded set $E_s$. Thus, it is Riemann integrable on $E_s$ \cite[\S 11.2.2, Thm 1]{Zorich}. Then, for all $x \in \X$,
$$
\widetilde{L_s}(x) = 
\underset{n \to \infty}{\lim} \sum_{i=1}^n(v^s_{i,n} - v^s_{i-1,n})p(v^s_{i,n})\mathbb{1}_{E_s}(v^s_{i,n})K(v^s_{i,n}, x) = \int_{E_s} K(x,y) p(y) dy    .
$$
Thus, by Lemma \ref{lemma:RKHSandPointwiseConvergence},
\begin{equation}
  \label{eq:Lntilde_def}  
\widetilde{L_s} = \int_{E_s} K(.,y) p(y) dy .
\end{equation}

From \eqref{eq:scalarf}, we obtain that for all $f \in \HS $, 
\[ 
L_sf = \underset{n \to \infty}{\lim} \sum_{i=1}^n(v^s_{i,n} - v^s_{i-1,n})p(v^s_{i,n})\mathbb{1}_{E_s}(v^s_{i,n}) f(v^s_{i,n}) = \langle f, \widetilde{L_s} \rangle .
\]
As this limit does not depend on the subdivision $(v^s_{i,n})_{0 \leq i \leq n, n \in \N}$, this implies that all functions of $\HS$ are $p$-integrable on $E_s$ and,
\[ \forall f \in \HS, \qquad \int_{E_s} f(x) p(x) dx =  \langle f, \widetilde{L_s} \rangle \]
In particular, for $t \in \N$, choosing $f = \widetilde{L_t}$ (which is in $\HS$) and using \eqref{eq:Lntilde_def}, we have
\begin{equation*} \label{eq:scalarnm}
\langle \widetilde{L_s}, \widetilde{L_t} \rangle  = \int_{E_s} \left( \int_{E_t} K(x,y)p(x) p(y) dy \right)   dx    .
\end{equation*}
By the same argument used above, the Riemann integral $\int_{E_s \times E_t}K(x,y) p(x)p(y) dx \, dy$ exists. By Fubini's Theorem for Riemann integrals \cite[\S 11.4]{Zorich}, we have
\begin{equation} \label{eq:scalarprodImproper}
\langle \widetilde{L_s}, \widetilde{L_t} \rangle  = \int_{E_s \times E_t } K(x,y)p(x)p(y) dx \, dy   .
\end{equation}

Let us show that when $s, t$ tend  to $+\infty$,
$\langle \widetilde{L_s}, \widetilde{L_t} \rangle$ converges to $\int_{\X \times \X}K(x,y) p(x)p(y) dx \, dy$. Let $s, t \in \N$,  as $x,y \mapsto |K(x,y)|p(x)p(y) $ is integrable on $\X \times \X$ and using Proposition \ref{prop:comp} with $f(x,y) =  K(x,y)p(x)p(y) \mathbb{1}_{\X \times \X \backslash E_s \times E_t }(x,y) $ and $g(x,y) = |K(x,y)|p(x)p(y)$ for all $x,y \in \X$, then $\int_{\X \times \X \backslash E_s \times E_t } K(x,y)p(x)p(y) dx dy $ and $\int_{\X \times \X \backslash E_s \times E_t } |K(x,y)|p(x)p(y) dx dy$ exist. Let us denote \( I(s,t)  = \int_{E_s \times E_t} K(x,y) p(x) p(y) \, dx \, dy
\)
and \( I  = \int_{\X \times \X} K(x,y) p(x) p(y) \, dx \, dy.
\)
We then have, 
\begin{align*}
    \left| I(s,t) - I \right| &= \left| \int_{\X \times \X \backslash E_s \times E_t }  K(x,y)  p(x)p(y) dx dy \right|  \\
    & \leq \int_{\X \times \X \backslash E_s \times E_t } \left| K(x,y) \right| p(x)p(y) dx dy  \\
    & = \int_{\X \times \X } \left| K(x,y) \right| p(x)p(y) dx dy - \int_{E_s \times E_t } \left| K(x,y) \right| p(x)p(y) dx dy  .
\end{align*}

As above, since \( E_{\min(s,t)} \times E_{\min(s,t)} \subset E_s \times E_t \), Proposition \ref{prop:comp} implies that 
\(
\int_{E_{\min(s,t)} \times E_{\min(s,t)}} |K(x,y)| p(x) p(y) \, dx \, dy
\)
exists and
\begin{align*}
\left| I(s,t) - I \right| 
&\leq \int_{\X \times \X} \left| K(x,y) \right| p(x) p(y) \, dx \, dy - \int_{E_{\min(s,t)} \times E_{\min(s,t)}} |K(x,y)| p(x) p(y) \, dx \, dy.
\end{align*}
As $\left(E_{s} \times E_{s}\right)_{s \in \N}$ is an exhaustion of $\X \times \X$, $\int_{E_{\min(s,t)} \times E_{\min(s,t)}} |K(x,y)| p(x) p(y) \, dx \, dy$ converges to $\int_{\X \times \X} |K(x,y)| p(x) p(y) \, dx \, dy $ as $s, t$ tend  to $+ \infty$. Thus, $I(s,t) = \langle \widetilde{L_s}, \widetilde{L_t} \rangle $ converges to $I = \int_{\X \times \X} K(x,y)p(x)p(y) dx \,dy$ as $s, t$ tend to $+\infty$ . \\
By applying Loève criterion \ref{loève} to the sequence \( (\widetilde{L_s})_{s \in \mathbb{N}} \), we conclude that \( (\widetilde{L_s}) \) converges in \( \HS \). Let \( \mu_p \) denote its limit. Notice that for all \( x \in \X \), we have
\[
\mu_p(x) = \lim_{s \to \infty} \int_{E_s} K(x,y) p(y) \, dy.
\]
Thus, for all $x \in \X,$ the sequence $(\int_{E_s} K(x,y) p(y) \, dy)$ converges as $s$ tends to $+ \infty $ and its limit does not depend on the exhaustion $(E_s)$. This implies that 
\( \int_{\X} K(x,y) p(y) \, dy \) is finite for all \( x \in \mathcal{X} \), and by Lemma \ref{lemma:RKHSandPointwiseConvergence}, we deduce that
\[
\mu_p = \int_{\mathcal{X}} K(.,y) p(y) \, dy.
\]
To conclude the proof of (a), recall that $I(s,t) \to I$ as $s$ and $t$ tend to $+ \infty$. This implies that $I(s,s) \to I$ when $s$ tends to $+ \infty$. As $\widetilde{L_s} $ converges in $\HS$ to $\mu_p$, we conclude that
\[
\Vert \mu_p \Vert^2=\int_{\X \times \X} K(x, y) p(x)p(y) dx \, dy.
\]
It remains to prove (b). Let $f \in \HS$. From the proof of (a), we have
$$
\int_{E_s} f(x) p(x)dx =  \langle f, \widetilde{L_s} \rangle   .
$$
As $(\widetilde{L_s})$ converges in $\HS$ to $\mu_p$ when $s$ tends to $+\infty$, we obtain that $(\int_{E_s} f(x) p(x) dx)$ converges  to $\langle f, \mu_p \rangle $ as $s$ tends to $+\infty$. As the limit does not depend on the chosen exhaustion, we conclude that $fp$ admits an improper integral on $\X$ and,
$$
\int_\X f(x)p(x)dx = 
\langle f,  \mu_p \rangle    .
$$
This implies that the mapping $f \in \HS \mapsto \int_\X f(x)p(x) dx$ is continuous. 
\end{proof}
\begin{remark}
Theorem \ref{thm:mean_embedding} contains two statements. The first one, (a), is that the function \({\mu_p : x \in \X \mapsto \int_\X K(x,y) p(y) \, dy}\) exists and belongs to \(\HS\). This result is the analogue for improper Riemann integrals of that established for Lebesgue integrals \cite[Proposition 2.3]{barp2024targeted} \cite{oates2022minimum}. Under the same condition, we further demonstrate in (b) that the reproducing property holds, with a less restrictive condition than in \cite{Mean_embedding, muandet2017kernel}, as we will show in Proposition \ref{prop:embed}. Notice that (b) is not simply deduced from (a) by continuity of the scalar product in $\HS$ since the interchange of limit and integral is not guaranteed in general.
\end{remark}

\begin{proposition} \label{prop:embed}
Let $\HS$ be a RKHS of real-valued function defined on $\X$, with reproducing kernel $K$. Assume that $K$ and $p$ are continuous almost everywhere and locally bounded on $\X \times \X$. If the standard variation \( \int_{\X} \sqrt{K(x, x)}p(x)dx \) exists, then \( \iint_{\X \times \X} |K(x, y)| p(x) p(y)dx \, dy \) is also finite.
\end{proposition}
\begin{proof} 
Let \( (E_n) \) be an exhaustion of \(\X\). 
Recall that
$K(x,y) = \langle K(x, .), K(y, .) \rangle_\HS$ for all $x, y \in \X$.
Using Cauchy-Schwarz inequality, we then obtain: 
\[
\int_{E_n \times E_n} |K(x, y)| p(x) p(y) \, dx \, dy \leq \int_{E_n \times E_n} \sqrt{K(x, x)} p(x) \sqrt{K(y, y)} p(y) \, dx \, dy.
\]
Since \( E_n \times E_n \) is an admissible subset of \( \X \times \X \), applying Fubini's theorem for Riemann integrals (cf. \cite[\S 11.4]{Zorich}) gives:
\begin{align*}
    \int_{E_n \times E_n} \sqrt{K(x, x)} p(x) \sqrt{K(y, y)} p(y) \, dx \, dy &= \int_{E_n} \int_{E_n} \sqrt{K(x, x)} p(x) \sqrt{K(y, y)} p(y) \, dx \, dy \\
    &= \left( \int_{E_n} \sqrt{K(x, x)} p(x) \, dx \right)^2.
\end{align*}

Thus, if \( \int_{\X} \sqrt{K(x, x)} p(x) \, dx \) is finite, it follows that \( \iint_{\X \times \X} |K(x, y)| p(x) p(y) \, dx \, dy \) is also finite. 
\end{proof}
The minimal requirement for the existence of the mean embedding $\mu_p$ is that  \( \iint_{\X \times \X} |K(x, y)| p(x) p(y)dx \, dy \) is finite. This is weaker than the criterion found in the literature that requires \( \int_{\X} \sqrt{K(x, x)}p(x)dx \) to be finite, as shown by the previous proposition. 
Note, however, that the (almost sure) continuity of the kernel is also required in Theorem~\ref{thm:mean_embedding}, which is a mild assumption.
Finally, the following example exhibits a kernel $K$ for which \( \int_{\X} \sqrt{K(x, x)}p(x)dx \) is not finite but such that \( \iint_{\X \times \X} |K(x, y)| p(x) p(y) dx \, dy\) is finite.
\newline
\begin{example}[Kernel with Non-Finite Standard Deviation and Well-Defined Mean Embedding]
Let \( p : x \mapsto \frac{1}{x^2}\mathbb{1}_{x \geq 1} \) and let us consider the following function:
\[
K: [1,\infty) \times [1,\infty) \to \mathbb{R}, \quad K(x, y) = xy e^{-\left(x - y\right)^2}.
\]
It is a positive semi-definite kernel as it is the covariance function of the process $xZ(x)$ where $Z$ is a stationary process with mean zero and Gaussian covariance function $\exp(-(x-y)^2)$.
The standard deviation is clearly infinite:
\[
\mathbb{E}_{X \sim p}[\sqrt{K(X, X)}]  = \int_1^{+\infty} \frac{1}{x} \, dx = +\infty.
\]
Let us now show that  \(
I := \int_{[1,\infty)^2} \vert K(x, y) \vert p(x) p(y) dx \, dy \) exists. \\
As the function $x,y \mapsto K(x, y) p(x) p(y)$ is positive, it is enough to prove that for only one specific exhaustion $(E_n)$ of the set $[1,\infty)$, the sequence \( I_n := \iint_{E_n ^2} K(x, y) p(x) p(y) dx \, dy\) converges to $I$.
Indeed, the positivity guarantees that the result will then be valid for all exhaustions \cite[\S 11.6, Prop. 1]{Zorich}.
Here, we choose  \( E_n = [ 1, n] \).\\ 
To compute $I_n$, let us split the cubic domain $E_n \times E_n$ into two triangles:
\begin{align*}
     I_n 
     &= \int_1^n \int_1^n \frac{1}{xy} e^{-\left( x - y \right)^2} \, dx dy 
     &= \int_1^n \int_y^n \frac{1}{xy} e^{-\left( x - y \right)^2} \, dx dy +  \int_1^n \int_1^y \frac{1}{xy} e^{-\left( x - y \right)^2} \, dx dy    .
\end{align*}
Using the symmetry of the integrand, we can see that the two terms are equal, and thus
\begin{align*}
     I_n = 2 \int_1^n \int_y^n \frac{1}{xy} e^{-\left( x - y \right)^2} \, dx dy  .
\end{align*}
Now let us apply the change of variables
\(
T : (x,y) \in \mathbb{R}^2 \mapsto (w = x - y, y)  
\).
As 
for all \( x, y \in \R \),
\[
1 \leq y \leq n  \quad \text{and} \quad  y \leq x \leq n \iff 0 \leq w \leq n - 1  \quad \text{and} \quad 
1 \leq y \leq n - w  ,
\]
we have
\begin{equation*}
   I_n = 2 \int_{0}^{n-1} \int_{1}^{n-w} \frac{1}{y(w + y)} e^{-w^2} \, dy  dw   .
\end{equation*}
Now, for all $w \in [0, n-1]$, we have
$$ 
0 \leq \int_{1}^{n-w} \frac{1}{y(w + y)} \, dy
\leq \int_{1}^{n-w} \frac{1}{y^2} \, dy 
= \left[ - \frac{1}{y} \right]_1^{n-w} 
= 1 - \frac{1}{n - w} \leq 1.
$$
Therefore,
$$ 
I_n \leq 2\int_0^{n-1}  e^{-w^2} \, dw \leq 2 \int_0^{+\infty} e^{-w^2} dw =\sqrt{\pi}.
$$
Finally, the sequence $(I_n)_{n \geq 1}$, which is monotonically increasing, 
is bounded. Hence it converges as \( n \to + \infty \). 
Therefore, 
\(
\iint_{[1, \infty)^2} \vert K(x, y) \vert p(x) p(y) dx \, dy
\)
is finite, even though \( \mathbb{E}_{X \sim p}[ \sqrt{K(X, X)}] \) is not finite.
\end{example}
\section*{Acknowledgement}
This research has been done in the frame of the Chair PILearnWater, part of the AI Cluster ANITI, funded by the French National Research Agency.

\bibliographystyle{ieeetr}
\bibliography{biblio}

\end{document}